\documentclass[12pt]{amsart}

\usepackage{amsthm,amsmath,mathrsfs}
\usepackage{amssymb,amsthm,amsmath}
\usepackage[dvips]{geometry}
\usepackage{amscd}
\usepackage[latin1]{inputenc}
\usepackage[all] {xy}

\usepackage{colortbl}
\usepackage{color,graphics}

\usepackage[dvips]{geometry}
\usepackage[dvips]{graphicx}
\usepackage{psfrag}

\theoremstyle{plain}
\newtheorem{mainthm}{Theorem}

\swapnumbers
\newtheorem{theorem}{Theorem}[section]
\newtheorem{proposition}[theorem]{Proposition}
\newtheorem{corollary}[theorem]{Corollary}
\newtheorem{lemma}[theorem]{Lemma}

\newtheorem{definition}[theorem]{Definition}

\newtheorem{remark}[theorem]{Remark}

\newcommand{\tcap}{\pitchfork}

\newcommand{\Z}{\mathbb{Z}}
\newcommand{\N}{\mathbb{N}}
\newcommand{\R}{\mathbb{R}}
\newcommand{\eps}{\varepsilon}

\newcommand{\SM}{{\mathcal M}}

\newcommand{\SR}{{\mathcal R}}
\newcommand{\ST}{{\mathcal T}}
\newcommand{\SU}{{\mathcal U}}

\newcommand{\diff}{\operatorname{Diff}}

\newcommand{\per}{\operatorname{Per}}

\newcommand{\fix}{\operatorname{Fix}}

\title{Mixing-like properties for some generic and robust dynamics}

\author{Alexander Arbieto}
\email{arbieto@im.ufrj.br}
\address{Instituto de Matem\'atica, Universidade Federal do Rio de Janeiro, P. O. Box 68530, 21945-970 Rio de Janeiro, Brazil.}
\author{Thiago Catalan}
\email{tcatalan@famat.ufu.br}
\address{Faculdade de Matem\'atica, Universidade Federal de Uberl\^andia.}
\author{Bruno Santiago}
\email{bruno\_santiago@im.ufrj.br}
\address{Instituto de Matem\'atica, Universidade Federal do Rio de Janeiro, P. O. Box 68530, 21945-970 Rio de Janeiro, Brazil.}

\date{\today}

\begin{document}

\maketitle

\begin{abstract}
We show that the set of Bernoulli measures of an isolated topologically mixing homoclinic class of a generic diffeomorphism is a dense subset of the set of invariant measures supported on the class.  For this, we introduce the large periods property and show that this is a robust property for these classes. We also show that the whole manifold is a homoclinic class for an open and dense subset of the set of robustly transitive diffeomorphisms far away from homoclinic tangencies. In particular, using results from Abdenur and Crovisier, we obtain that every diffeomorphism in this subset is robustly topologically mixing. 
\end{abstract}

\section{introduction}

The study of chain-recurrence classes began once that Conley's Fundamental Theorem of Dynamical Systems appeared. It says that up to quotient these classes on points any dynamical system looks like a gradient dynamics. 

However, some of these classes, called homoclinic classes, gained interest with the advent of Smale's Spectral Decomposition Theorem. Indeed, this theorem says that for Axiom A (hyperbolic) dynamics the non-wandering set splits into finitely many homoclinic classes. Moreover, each of these classes is isolated: it is the maximal invariant set of a neighbourhood of itself. 
Thus, these homoclinic classes are the sole chain recurrence classes of such dynamics. We recall that a homoclinic class of a periodic point $p$ is the closure of the transversal intersections of the invariant manifolds of the orbit of $p$. It is well known that such classes are transitive, i.e. they contain a point whose orbit is dense in the class.

Hence, the study of homoclinic classes, in non-hyperbolic situations, attracted the attention of many mathematicians, see \cite{BDV} for a survey on the subject. The purpose of this article is to contribute to this study both in the measure theoretical viewpoint and the topological viewpoint. 
The dynamical systems we shall consider here are diffeomorphisms and the topology 
used in the space of diffeomorphisms will be the $C^1$-topology.

In ergodic theory, an important problem is to describe the set of invariant measures of a dynamical system, since the theory says that the invariant measures help to describe the dynamics. In \cite{S2}, Sigmund studied this problem in the hyperbolic case. 
More precisely he proved that for any homoclinic class of an Axiom A diffeomorphism, the set of periodic measures, i.e. Dirac measures evenly distributed on a periodic orbit, is dense in the set of invariant measures. 
On the other hand, there is a refinement of the Spectral Decomposition Theorem, due to Bowen, which says that any such class of an Axiom A system splits into finitely many compact sets which are cyclically permuted by the dynamics and the dynamics of each piece, at the return, is topologically mixing, i.e. given two open sets $U$ and $V$ then the $n$-th iterate of $U$ meets $V$ \emph{for every} $n$ large enough. 
Using this, Sigmund in \cite{S1} was able to prove that the set of Bernoulli measures is dense among the invariant measures. He also proved that weakly mixing measures contains a residual subset of invariant measures. Indeed, the set of weakly mixing measures is a countable intersection of open sets. We recall that a measure is Bernoulli if the system endowed with it is measure theoretically isomorphic to a Bernoulli shift.

In the non-hyperbolic case, \cite{ABC} proved that for a generic diffeomorphism any isolated homoclinic class has periodic measures dense in the set of invariant measures, thus extending the first result of Sigmund mentioned above to the generic setting. We recall that a property holds generically if it holds in a countable intersection of open and dense sets of diffeomorphisms. Our first result extends the second result of Sigmund mentioned above.

\begin{mainthm}
\label{teoA}
For any generic diffeomorphism $f$, if the dynamics restricted to an isolated homoclinic class is topologically mixing then the Bernoulli measures are dense in space of invariant measures supported on the class. In particular, the set of weakly mixing measures contain a residual subset.
\end{mainthm}

 The main tools employed here to prove Theorem \ref{teoA} are the results
from \cite{ABC}, mentioned above, the main theorem in \cite{AC}, and the {\it large 
periods property} which is a tool that we devised in order to detect mixing behavior.
For instance, a dynamical system has {\it large periods property} if there are periodic
points with any large enough period which are arbitrarily dense. The presence of this
property
implies that the system is topologically mixing. In the differentiable setting, we 
also define the \emph{homoclinic large periods property} which only considers the homoclinically related periodic points. We prove that this property is robust, see Proposition \ref{r.lpp}. We recall that a property is robust if it holds in a neighbourhood of the diffeomorphism. 

In \cite{AC}, the authors use their main result to prove that any homoclinic class of a generic diffeomorphism has a spectral decomposition in the sense of Bowen, like discussed before. One of the motivations is that all known examples of robustly transitive diffeomorphisms are robustly topologically mixing.

So, in the same article the authors ask the following questions:

\begin{enumerate}
 \item Is \emph{every} robustly transitive diffeomorphism topologically mixing?
 \item Failing that, is topological mixing at least a $C^1$ open and dense condition within the space of all robustly transitive diffeomorphisms?
\end{enumerate}

Now, we point out that  the results of section 2 of  \cite{AC} gives immediately the following result\footnote{We would like to thank Prof. Sylvain Crovisier for pointing out this result to us.} (see also Remark \ref{mixing1}).

\begin{theorem}
\label{baranasquestion}
Let $f$ be a generic diffeomorphism. If an isolated homoclinic class of $f$ is topologically mixing then it is robustly topologically mixing.
\end{theorem}

Actually, since the large periods property implies topological mixing, the robustness of this property could lead to another proof of the previous result, see Section 4.

We want now attack problem (2) above. To this purpose it is natural to look for the global dynamics of the previous theorem instead of the semi-local dynamics. This leads us to a question posed in \cite{BDV} (Problem 7.25, page 144): ``For an open and dense subset of robustly transitive partially hyperbolic diffeomorphism: Is the whole manifold robustly a homoclinic class?''. 
Recall by a result of \cite{BC}, for generic transitive diffeomorphisms, the whole manifold is a homoclinic class.

Our next result gives a positive answer to Problem 7.25 of \cite{BDV} (quoted above)
\emph{far from homoclinic tangencies}. A homoclinic tangency is a non-transversal intersection between the invariant manifolds of a hyperbolic periodic point.
The result is the following:

\begin{mainthm}
 There exists an open and dense subset  among robustly transitive diffeomorphisms far from homoclinic tangencies  formed by diffeomorphisms such that the whole manifold is a homoclinic class.
\label{main.homoclinic}\end{mainthm}

 This result together with Theorem \ref{baranasquestion} give us a partial answer to question (2) above, posed in \cite{AC}. 

\begin{mainthm}
\label{r.mixing}
There is an open and dense subset among robustly transitive diffeomorphisms far from homoclinic tangencies formed by robustly topologically mixing diffeomorphisms. 
\end{mainthm}

These two results were previously obtained by \cite{BDU} for strongly partially hyperbolic diffeomorphisms 
with one dimensional center bundle, see also \cite{HHU}. By strong partial hyperbolicity we mean partial hyperbolicity with both non-trivial extremal bundles such that the center bundle splits in one-dimensional subbundles in a dominated way.
Actually, they obtain this proving that one of the strong foliations given by the partial hyperbolicity is minimal, which is 
a stronger property than topological mixing. In order to obtain this minimality they used arguments involving the accessibility property. We notice however that our results hold for diffeomorphisms with higher dimensional center directions. In section two, we present a way to produce such examples.

 This paper is organized as follows. In Section 2 we give the precise definitions of the main objects we shall
deal with. In Section 3 we state the known results that will be our main tools. 
In Section 4 we introduce the large periods property.  
In Section 5 we use the large periods property to prove Theorem \ref{teoA}. 
Finally, in Section 6 we prove Theorem \ref{main.homoclinic}.

\

{\bf Acknowledgements:} A.A. and B.S. want to thank FAMAT-UFU, and T.C. wants to thank IM-UFRJ for the kind hospitality of these institutions during the preparation of this work. This work was partially supported by CNPq, CAPES, FAPERJ and FAPEMIG.

\section{Precise Definitions}

In this section, we give the precise definitions of the objects used in the statements of the results.  In this paper $M$ will be a closed and connected Riemaniann manifold of dimension $d$. Also, $cl(.)$ will denote the closure operator.

\subsection{Topological dynamics}

Let $f:M\rightarrow M$ be a homeomorphism. Given $x\in M$, we define the \emph{orbit} of $x$ 
as the set $O(x):=\{f^n(x);n\in\Z\}$. The forward orbit of $x$ is the set $O^{+}(x):=\{f^n(x);n\in\N\}$. 
In a similar way we define the backward orbit $O^{-}(x)$. If necessary, to emphasize the dependence of $f$,
we may write $O_f(x)$.

Given $\Lambda\subset M$ we say that it is an \emph{invariant} set if
$f(\Lambda)=\Lambda$.

We recall the notions of transitivity and mixing. We say that $f$ is transitive if there exists a point
in $M$ whose forward orbit is dense. This is equivalent to the existence of a dense backward orbit and
is also equivalent to the following condition: for every pair $U,V$ of open sets, there exists $n>0$ 
such that $f^n(U)\cap V\neq\emptyset$.

More specially, we say that $f$ is topologically mixing if for every par $U,V$ of open sets there 
exists $N_0>0$ such that $n\geq N_0$ implies $f^n(U)\cap V\neq\emptyset$.

\subsection{Hyperbolic Periodic Points}

We say that $p$ is a \emph{periodic point} if $f^n(p)=p$ for some $n\geq 1$. The minimum of such $n$ is called the \emph{period} of $p$ and it is denoted by $\tau(p)$.

 The periodic point is \emph{hyperbolic} if the eigenvalues of $Df^{\tau(p)}(p)$ do not belong to $S^1$. As usual, $E^s(p)$ (resp. $E^ u(p)$) denotes the eigenspace of the eigenvalues with norm smaller (resp. bigger) than one. 
This gives a $Df^{\tau(p)}$ invariant splitting of the tangent bundle over the orbit $O(p)$ of $p$. The {\it index } of a hyperbolic periodic point $p$ is the dimension of the stable direction, denoted by $I(p)$.

If $p$ is a hyperbolic periodic point for $f$ then every diffeomorphism $g$, $C^1-$close to $f$ have also a hyperbolic periodic point close to $p$ with same period and index, which is called the {continuation of $p$ for $g$}, and it is denoted by {\it $p(g)$}.

The local stable and unstable manifolds of a hyperbolic periodic point $p$ are defined as follows: given $\eps>0$ small enough, we set
$$
W^s_{loc}(p)=\{x\in M ; \ \ d(f^n(x), f^n(p))\leq \eps \text{, for every }n\geq 0\} \text{ and } $$
$$W^u_{loc}(p)=\{x\in M ; \ \d(f^{-n}(x), f^{-n}(p))\leq \eps \text{, for every }n\geq 0\}.
$$ 
They are differentiable manifolds tangent at $p$ to $E^s(p)$ and $E^u(p)$. The stable and unstable manifolds are given by the saturations of the local manifolds. indeed,
$$W^s(p)=\bigcup_{n\geq 0}f^{-n\tau(p)}(W^s_{loc}(p))\text{  and }W^u(p)=\bigcup_{n\geq 0}f^{n\tau(p)}(W^u_{loc}(p)).$$

The stable and unstable set of a hyperbolic periodic orbit, $O(p)$ are given by:
$$
W^s(O(p))=\bigcup_{j=0}^{\tau(p)-1} W^s(f^j(p)) \text{ and } W^u(O(p))=\bigcup_{j=0}^{\tau(p)-1} W^u(f^j(p)).
$$

\subsection{Homoclinic Intersections}

If $p$ is a hyperbolic periodic point of $f$, then its \emph{homoclinic class} $H(p)$ is the closure of the transversal intersections of the stable manifold and unstable manifold of the orbit of $p$:
$$H(p)=cl\big(W^s(O(p))\pitchfork W^u(O(p))\big).$$

We say that a hyperbolic periodic point $q$ is \emph{ homoclinically related} to $p$ if $W^s(O(p))\pitchfork W^u(O(q))\neq \emptyset$ and $W^u(O(p))\pitchfork W^s(O(q))\neq \emptyset$. It is well known that a homoclinic class coincides with the closure of the hyperbolic periodic points homoclinically related to $p$. Moreover, it is a  transitive invariant set.  We say that a homoclinic class $H(p)$ has a robust property if $H(p(g))$ has also this property for any diffeomorphism $g$ sufficiently close to $f$. 

We define {\it the period of a homoclinic class} $H(p)$ as the greatest common divisor
of the periods of the hyperbolic periodic points homoclinically related to $p$ and we denote by $l(O(p))$.

We say that the homoclinic class $H(p)$ is isolated if there exists a neighbourhood $U$ of $H(p)$ such that $H(p)=\bigcap_{n\in \mathbb{Z}}f^n(U)$.

On the other hand, we say that a non-transversal intersection between $W^s(O(p))$ and $W^u(O(p))$ is a  \emph{homoclinic tangency}. We denote by $\mathcal{HT}(M)$ the set of diffeomorphisms exhibiting a homoclinic tangency. We will say that a diffeomorphism $f$ is \emph{far from homoclinic tangencies} if $f\notin cl(\mathcal{HT}(M))$.

Given $p$ and $q$ hyperbolic periodic points with $I(p)<I(q)$ we say that they form a {\it heterodimensional
cycle} if there exists $x\in W^s(O(p))\cap W^u(O(q))$, with $\dim{(T_xW^s(O(p))\cap T_xW^u(O(q)))}=0$ 
and $W^u(O(p))\tcap W^u(O(q))\neq\emptyset$.

\subsection{Invariant Measures}

A probability measure $\mu$ is $f$-invariant if $\mu(f^{-1}(B))=\mu(B)$ for every measurable set $B$. An invariant measure is ergodic if the measure of any invariant set is zero or one. Let $\SM(f)$ be the space of $f$-invariant \textit{probability measures} on $M$, and let $\SM_e(f)$ denote the ergodic elements of $\SM(f)$.

For a periodic point $p$ of $f$ with period $\tau(p)$, we let $ \mu_p$ denote the  periodic measure  associated to $p$, given by
$$\mu_p =\frac{1}{\tau(p)} \sum_{ x\in O(p)} \delta_x $$
where $\delta_x$ is the Dirac measure at $x$.

Given an invariant measure $\mu$, Oseledets' Theorem says that for almost every $x\in M$ and all $v\in T_xM$ the limits
$$\lambda(x,v):=\lim_{n\to^{+}_{-}\infty}\frac{1}{n}\log{\|Df^n(x)v\|}$$ 
exists and are equal. Moreover, one has a measurably varying splitting of the tangent bundle $TM=E_1\oplus...\oplus E_k$ and measurable invariant functions $\lambda_j:M\to\R$, $j=1,...,k$ such that if $v\in E_j$ then $\lambda(x,v)=\lambda_j(x)$. The number $\lambda_j(x)$ is called the \emph{Lyapunov exponent} of $f$ at $x$.

Now, let us define the notion of Bernoulli measure. We first recall the so-called Bernoulli shift. It is the homeomorphism $\sigma:\{1,...,n\}^{\Z}\to\{1,...,n\}^{\Z}$ defined by $\sigma(\{x_{n}\})=\{x_{n+1}\}$. 
In $\{1,...,n\}^{\Z}$ consider $m_B$ the product measure with respect to the uniform probability in $\{1,...,n\}$.
It is easy to see that $m_B$ is invariant under $\sigma$.

We say that $\mu\in\SM(f)$ is a \emph{Bernoulli measure} if $(f,\mu)$ is measure theoretically isomorphic to $(\sigma,m_B)$.

\subsection{Partial hyperbolicity}

Let $\Lambda\subset M$ to be invariant under a diffeomorphism $f$. Let $E,F$ to be subbundles of $T_{\Lambda}M$ of the tangent bundle over
$\Lambda$ with trivial intersection at every $x\in\Lambda$. We say that $E$ \emph{dominates} $F$ if there exists $N\in\N$ such that 
$$\|Df^N(x)|_E\|\|Df^{-N}(f^N(x))|_F\|\leq\frac{1}{2},$$
for every $x\in\Lambda$.
We say that $\Lambda$ admits a \emph{dominated
splitting} if there exists a decomposition of the tangent bundle $T_{\Lambda}M=\bigoplus_{l=1}^kE_l$ such that $E_l$ dominates $E_{l+1}$.

We say that a $f$-invariant subset $\Lambda$ is {\it partially hyperbolic} if it admits a dominated splitting $T_{\Lambda}M=E^s\oplus E^c_1\oplus\ldots\oplus E^c_k\oplus E^u$,
with at least one of the extremal bundles being non-trivial,
such that the extremal bundles have uniform contraction and expansion: there exist a constants $m\in\N$ such that for 
every $x\in M$:
\begin{align*}
\bullet \ &\|Df^m(x)v\|\leq 1/2 \text{  for each unitary } v\in E^s ,
\\
\bullet \ &\| Df^{-m}(x)v\|\leq 1/2 \text{  for each unitary } v\in E^u \\
\end{align*}
and the other bundles, which are called center bundles, do not contracts neither expands.

If all center bundles are trivial, then $\Lambda$ is called a {\it hyperbolic set}. 
Now, we say $\Lambda$ is {\it strongly partially hyperbolic} if  both extremal bundles and center bundle are non-trivial and moreover such that all of its center bundles are one-dimensional. In particular a strongly partially hyperbolic set is not hyperbolic.

We say that a diffeomorphism $f:M\rightarrow M$ is {\it partially hyperbolic} (resp. {\it strong partially hyperbolic} ) if $M$ is a partially hyperbolic (resp. strongly partially hyperbolic) set of $f$.  When $M$ is a hyperbolic set we say that $f$ is {\it Anosov}.

We remark now that strongly partially hyperbolic diffeomorphisms are by definition 
far from homoclinic tangencies, since all central sub bundles have dimension one.

Examples of partially hyperbolic diffeomorphisms with higher dimensional central 
directions can be given by deforming some linear Anosov diffeomorphisms as in Ma\~n\'e's 
example. For instance, let $A$ be a linear Anosov diffeomorphism with eigenvalues 
$\lambda_1<\lambda_2<\lambda_3<1<\lambda_4$ but, such that $\lambda_2$ and $\lambda_3$ are close to 1. Then we can create a pitchfork bifurcation, producing two fixed points $p$ and $q$ with eigenvalues $\mu_1(p)<1<\mu_2(p)<\mu_3(p)<\mu_4(p)$ and $\mu_1(q)<\mu_2(q)<1<\mu_3(q)<\mu_4(q)$, such that $\mu_3(q)$ is still close to 1. Moreover, as in Ma\~n\'e's argument \cite{M} we can guarantee that this diffeomorphism is transitive. Now we can perform another pitchfork bifurcation on $q$ producing two other fixed points $q_1$ and $q_2$ with eigenvalues $\mu_1(q_1)<\mu_2(q_1)<1<\mu_3(q_1)<\mu_4(q_1)$ and $\mu_1(q_2)<\mu_2(q_2)<\mu_3(q_2)<1<\mu_4(q_2)$.  Once again, this diffeomorphism is transitive. Now, since the bifurcations preservers the center unstable leaves, we can guarantee that there exists a dominated splitting $E^s\oplus E_1^c\oplus E_2^c\oplus E^u$, where $E^c_1$ is related to $\mu_2$ and $E^c_2$ is related to $\mu_3$. As in Ma\~n\'e's example, the unstable foliation will be minimal. In particular, it will be topologically mixing also.

\begin{remark}\label{r.hps} If $f$ is partially hyperbolic, by Theorem 6.1 of \cite{HPS} there exist strong stable and strong unstable foliations that integrate $E^s$ and $E^u$. More, precisely, for any point $x\in M$ there is a unique invariant local strong stable manifold $W_{loc}^{ss}(x)$ which is a smooth graph of a function $\phi_x: E^s\rightarrow E^c\oplus E^u$ (in local coordinates), and varies continuously with $x$. In particular, $W_{loc}^{ss}(x)$ has uniform size for every $x\in M$. The same holds for $W_{loc}^{uu}(x)$, integrating $E^u$.

Saturating these local manifolds, we obtain two foliations, that we denote by $\mathcal{F}^s$ and $\mathcal{F}^u$ respectively. Indeed, $\mathcal{F}^s(x)=\bigcup_{n\geq 0}f^{-n}(W^{ss}_{loc}(f^n(x))$. Analogous definition holds for $\mathcal{F}^{u}$.

\end{remark}

\subsection{Robustness and Genericity}

As mentioned before, we deal with the space $\diff^1(M)$ of $C^1$ diffeomorphisms over $M$ endowed with the $C^1$-topology. This is a Baire space. Thus any residual subset, i.e. a countable intersection of open and dense sets, is dense. When a property $P$ holds for any diffeomorphism in a fixed residual subset, we will say that $P$ holds generically. Or even, that a generic diffeomorphisms exhibits the property $P$.

On the other hand, we say that a property holds robustly for a diffeomorphism $f$ if there exists a neighbourhood $\mathcal{U}$ of $f$ such that the property holds for any diffeomorphism in $\mathcal{U}$. 

In this way, we say that a diffeomorphism $f\in\diff^1(M)$ is \emph{robustly transitive} if it admits
a neighborhood entirely formed by transitive diffeomorphisms. 

In this paper we let $\ST(M)$ denote the open set of $\diff^1(M)$ formed by robustly transitive diffeomorphisms
which are far from tangencies. Notice that being far from tangencies is, by definition, an open condition.
Also we define by $\mathcal{T}_{NH}(M)$ as the interior of robustly transitive strongly partially hyperbolic diffeomorphisms, which is a subset of  $\mathcal{T}(M)$. 

When dealing with properties which involves objects defined by the diffeomorphism itself we need to deal with the continuations of these objects.

For instance, when we say that a homoclinic class of $f$ is robustly topologically mixing, we are fixing a hyperbolic periodic point $p$ of $f$ and a neighbourhood $\mathcal{U}$ of $f$ such that for any $g\in  \mathcal{U}$ the continuation $p(g)$ of $p$ is defined and the homoclinic class $H(p(g),g)$ is topologically mixing, i.e. for any $U$ and $V$ open sets of $H(p(g),g)$ there exists $N>0$ such that for any $n\geq N$ we have $g^n(U)\cap V\neq \emptyset$.

Another example of a robust property is given by the following well known result which says that partial hyperbolicity is a robust property.

\begin{proposition}[p. 289 of \cite{BDV}]
\label{p.phpersiste}
Let $\Lambda$ be a (strongly) partially hyperbolic set for $f$. Then, there exists a neighborhood $U$ of $\Lambda$ and a $C^1$
neighborhood $\SU$ of $f$ such that every $g$-invariant set $\Gamma\subset U$, is (strongly) partially hyperbolic,  for every $g\in\SU$. 
\end{proposition}

\section{Some Tools}

In this section, we collect some results that will be used in the proofs of the main results. 

\subsection{Perturbative Tools}

We start with Franks' lemma \cite{F}. This lemma enable us to deal with some non-linear problems using linear arguments.

\begin{theorem}[Franks lemma]
\label{l.franks}
Let $f\in \diff^1(M)$ and $\mathcal{U}$ be a $C^1$-neighborhood of $f$ in $\diff^1(M)$. 
Then, there exist a neighborhood $\mathcal{U}_0\subset \mathcal{U}$ of $f$ and $\delta>0$
such that if $g\in \mathcal{U}_0(f)$, $S=\{p_1,\dots,p_m\}\subset M$ and 
$\{L_i:T_{p_i}M\to T_{g(p_{i})}M\}_{i=1}^{m}$ are linear maps 
satisfying $\|L_i-Dg(p_i)\|\leq\delta$ for $i=1,\dots m$ then there exists 
$h\in\mathcal{U}(f)$ coinciding with $g$ outside any prescribed neighborhood of $S$ and 
such that $h(p_i)=g(p_i)$ and $Dh(p_i)=L_i$.
\end{theorem}

One of the main applications of Franks lemma is to change the index of a periodic orbit, after a perturbation, if the Lyapunov exponents of  the orbit is weak enough.  More precisely, we can prove the following:

\begin{lemma}
Let $f\in \diff^1(M)$ having a sequence of hyperbolic periodic points $p_n$ with some
index $s+1$, having negative Lyapunov exponents arbitrarily close to zero. Then, 
there exists $g$ arbitrarily close to $f$ having hyperbolic periodic points of indices $s$ and $s+1$. 
\label{change.index}\end{lemma}

{\it Proof:} Given a neighborhood $\mathcal{U}$ of $f$ let us consider $\delta>0$ given for this neighborhood and $U_0$ another small enough neighborhood of $f$. We will suppose that the sequence of periodic points $p_n$ is such that the smallest eigenvalue $\lambda_{p_n}$ of $Df^{\tau(p_n)}$ with absolute value smaller than 1, has multiplicity one. The argument is similar in the other cases. 

Our hypothesis says that
$$
\frac{1}{\tau(p_n)}\log \|Df^{\tau(p_n)}|E^s(p_n)\|=\frac{1}{\tau(p_n)}\log |\lambda_{p_n}|
$$    
approaches zero as $n$ grows.  
Now, let us consider  $E_n$ as the eigenspace of the eigenvalues $\lambda_{p_n}$, 
and $\{E_l\}$ the other eigenspaces.  
We can define linear maps $L_i:T_{f^i(p)}M\rightarrow T_{f^{i+1}(p)}M$, equal to $Df(f^i(p))$
in all subspaces $Df^i(p)E_l$, but in $Df^i(p)E_n$ we choose $L_i$ satisfying 
$\|L_i | \ Df^iE_n\|=(1+\alpha)\|Df(f^i(p))| Df^iE_n\|$, where $\alpha>0$ depends on
$\delta>0$. Then, $L_i$ is $\delta-$close to $Df(f^i(p))$, 
and also preserves the eigenspace $Df^i(p)E_n$.

Hence, using Franks lemma we can find $g\in \mathcal{U}$ such that $p_n$ still is a hyperbolic periodic point and moreover $Dg(f^i(p))=L_i$, where $g$ depends on the periodic point $p_n$. In particular, $E_n$ is a invariant subspace of $T_{p_n}M$ for $Dg^{\tau(p_n)}$ and moreover:
$$
\|Dg^{\tau(p)}(p_n)| E_n\|=(1+\alpha)^{\tau(p_n)}\lambda_n.
$$

Hence, by hypothesis, we can choose $p$ equal some $p_n$,  in order to have, after the above perturbation:
$$
\frac{1}{\tau(p)}\log \|Dg^{\tau(p)}|E_n(p)\|>0. 
$$ 
Since $L_i$ can be chosen such that the other Lyapunov exponents of  $p$ keep unchanged, we have that $p$
has index $s$. To finish the proof, we just observe that, Franks lemma changes the initial diffeomorphism only in a arbitrary neighborhood of the orbit of $p$, therefore the neighborhood $\mathcal{U}$ could be chosen such that the hyperbolic periodic point $p_1$ of $f$ has a continuation, which implies that $p_1(g)$ is also a hyperbolic periodic point of $g$ with index $s+1$. 
$\hfill\square$

Another result that we shall use is Hayashi's connecting lemma \cite{H}. This will be helpful to create some heterodimensional cycles.

\begin{theorem}[$C^1$-connecting lemma]
Let $f \in \diff^1(M)$ and $p_1,\, p_2$ hyperbolic periodic points of $f$, such that there exist sequences  $y_n\in M$ and positive integers 
$k_n$ such that:
\begin{itemize}
\item $y_n\rightarrow y \in W_{loc}^u(p_1, f))$, $y\neq p_1$; and
\item $f^{k_n}(y_n)\rightarrow x \in W_{loc}^s(p_2, f))$, $x\neq p_2$.
\end{itemize}
Then, there exists a $C^1$  diffeomorphism  $g$   $C^1-$close to $f$ such that $W^u(p_1,g)$ and $W^s(p_2,g)$ have a non empty intersection close to  $y$. \label{connecting lema}\end{theorem}

As it is well known, this result implies that if $f$ is a generic diffeomorphism having a non-hyperbolic 
homoclinic class which contains two periodic points $p$ and $q$ with different indices then 
there exist arbitrarily small perturbations of $f$ such that $p$ and $q$ belongs to a heterodimensional cycle.

\subsection{Generic Results}

We start this subsection with one of the main generic result used in this paper, which is a result 
of Abdenur and Crovisier, Theorem 3 in \cite{AC}. 
They prove the existence of a decomposition of any generic isolated chain-transitive set.
Since we solely are interested here in the study of isolated homoclinic classes, 
we quote their result only for homoclinic classes.  

\begin{theorem}[Theorem 3 in \cite{AC}] 
There exists a residual subset $\mathcal{R}\subset \diff^1(M)$
such that for every $f\in\SR$, any isolated homoclinic class $H(p,f)$ of a hyperbolic periodic point $p$ of $f$, decomposes uniquely as the finite union $H(p)=\Lambda_1\cup \ldots \cup \Lambda_l$, of disjoint compact sets on each of which $f^l$ is topologically mixing. Moreover, $l$ is the smallest positive integer such that $W^u(f^l(p))$ has a non empty transversal intersection with $W^s(p)$.
\label{t.acmixing}
\end{theorem}

As an application, they obtain that generically any transitive diffeomorphism is topologically mixing.

The positive integer $l$ in the previous theorem is, in fact, the period of the homoclinic class,  $l(O(p)).$ 
This number gives a nice information about the intersections between stable and unstable manifolds of hyperbolic  periodic points homoclinically related to $p$. More precisely:

Since for any two periodic points $p_1$ and $p_2$ which are homoclinically related their homoclinic classes $H(p_1)$ and $H(p_2)$ are equal we can recast the following result of \cite{AC} as:

\begin{proposition}[Proposition 2.1 in \cite{AC}]
Consider hyperbolic periodic points $p_1$ and $p_2$ which are homoclinically related to $p$, and such that $W^u(p_1)\pitchfork W^s(p_2)\neq \emptyset$. Then $W^u(f^n(p_1))\pitchfork W^s(p_2)\neq \emptyset$ if, and only if, $n$ belongs to the group $l(O(p))\Z$. 
\label{prop.bla}\end{proposition}

\begin{remark} In particular, if $\tilde{p}$ is homoclinically related to $p$, then $W^u(f^n(\tilde{p}))\pitchfork W^s(\tilde{p})\neq \emptyset$ if, and only if, $n\in l(O(p))\Z$. 
\label{rmk.int.}\end{remark}

Here, we also investigate properties of topologically mixing homoclinic classes which may not be the
whole manifold. In this sense we remark the following:

\begin{remark}
Also as a direct consequence of the Theorem \ref{t.acmixing} we have that generically, if an isolated
homoclinic class $H(p)$ is topologically mixing then $W^u(f(p))$ has a non empty transversal intersection
with $W^s(p)$.  Now, since this intersection is robust we point out that Theorem \ref{baranasquestion} is a consequence of this and Proposition 2.3 in \cite{AC}.   
\label{mixing1}\end{remark}

The result below, of Bonatti and Crovisier \cite{BC}, proves that a large class
of transitive diffeomorphism have the property that the whole manifold coincides 
with a homoclinic class.

\begin{theorem}[Bonatti and Crovisier]
\label{BC}
There exists a residual subset $\mathcal{R}$ of $\diff^1(M)$ such that for every 
transitive diffeomorphism $f\in\SR$ 
 if $p$ is a hyperbolic periodic point of $f$  then $M=H(p,f)$.
\end{theorem}
 
Another generic result is the following 
 \begin{theorem}[Theorem A, item (1), \cite{CMP}] 
 \label{t.cmp}
 There exists a residual subset $\mathcal{R}$ of $\diff^1(M)$ such that for every $f\in \SR$ if two homoclinic classes $H(p_1,f)$ and $H(p_2,f)$ are either equal or disjoint.
\end{theorem}

The next result, from \cite{ABCDW}, says that generically, homoclinic classes are index complete. 

\begin{theorem}[Theorem 1 in \cite{ABCDW}]
There is a residual subset $\mathcal{R}\in \diff^1(M)$ of diffeomorphisms $f$ such 
that, every $f\in \mathcal{R}$ and any homoclinic class containing hyperbolic periodic points of indices $i$ and $j$, also contains hyperbolic periodic points of index $k$ for every $i\leq k\leq j$.
\label{t.ABCDW}\end{theorem}

The next tool we shall use is due to  Abdenur, Bonatti and Crovisier in \cite{ABC} which extends Sigmund's result \cite{S2} to the non-hyperbolic setting.

\begin{theorem}[Theorem 3.5, item (a), in \cite{ABC}] Let $\Lambda$ be an isolated non-hyperbolic transitive set of a $C^1-$generic diffeomorphism $f$, then the set of periodic measures supported in $\Lambda$ is a dense subset of the set $M_f(\Lambda)$ of invariant measures supported in $\Lambda$. 
\label{ABC} \end{theorem}

Crovisier, Sambarino and Yang in \cite{CSY} showed that for any diffeomorphism $f$ in an open and dense subset far from homoclinic tangencies, every homoclinic class of $f$  has a kind of  strong partial hyperbolicity. More precisely, the difference is that the ``partially hyperbolic splitting'' found by them could have either one or both trivial extremal bundles. In this last scenario, by our definition the diffeomorphism would not be partially hyperbolic. However, by an abuse of notation, we will continue calling  it  partially hyperbolic as in \cite{CSY}. Their result gives other important properties. Like, information of the minimal and maximal indices of periodic points inside the homoclinic class. More precisely:

\begin{theorem}[Theorem 1.1(2) in \cite{CSY}]
There is an open and dense subset $\mathcal{A}\subset \diff^1(M)-\{cl(\mathcal{HT})\}$ such that for every $f\in \mathcal{A}$, any homoclinic class $H(p)$ is a partially hyperbolic set of $f$
$$
T_{H(p)}M= E^s\oplus E^c_1\oplus  \ldots E^c_k\oplus E^u,
$$
with $dim\ E^c_i=1$, $i=1,\ldots, k$, and moreover the minimal stable dimension of the periodic points of $H(p)$ is $dim(E^s) $ or $dim (E^{s})+1$. Similarly the maximal stable dimension of the periodic orbits of $H(p)$ is $dim(E^s)+k$ or $dim(E^s)+k-1$. For every $i$, $1\leq i\leq k$ there exists periodic points in $H(p)$ whose Lyapunov exponent along $E^c_i$, is arbitrary close to $0$.  
\label{p.CSY}\end{theorem}

\section{Large Periods Property}\label{lpp}

In this section we introduce the large periods property,  our main tool to detect mixing properties.

\begin{definition}
Let $f:X\to X$ be a homeomorphism of a metric space. We say that $f$ has the large periods property if for any $\eps>0$ there exists $N_0\in\N$ such that for every $n\geq N_0$ there exists $p_n\in\fix(f^n)$, whose orbit under $f$ is $\eps$ dense in $X$.
\end{definition}

A simple remark is that if $X$ has an isolated point and $f$ has the large period property then $X$ is a singleton.

The large periods property can be used as a criterion to assure mixing, as the next result shows.

\begin{lemma}
\label{criteriopramixing}
Every homeomorphism of a metric space with the large periods property is topologically mixing
\end{lemma} 
\begin{proof}
Let $f:X\to X$ be a homeomorphism with the large periods property. Notice that $f$ is transitive. Indeed, given $U$ and $V$ non-empty and disjoint open sets take $\eps<\min\{diam(U),diam(V)\}$. By the large periods property, there exists a point $p\in\per(f)$ whose orbit is $\eps$ dense in $X$. This implies that there exists a point $y\in V$ and $n>0$ such that $f^n(y)\in U$. Thus $f$ is transitive.

We now prove that $f$ is topologically mixing. Let $U$ and $V$ be non-void and disjoint open sets. By the transitivity of $f$ there exists a first iterate $n_1$ such that $f^{n_1}(U)\cap V\neq\emptyset$. In particular, $f^j(U)\cap V=\emptyset$ for every $j=1,...,n_1-1$. Take an open ball $B\subset U$, satisfying $$f^{n_1}(B)\subset f^{n_1}(U)\cap V,$$ and $\eps=diam(B)/2$. Let $N_0=N_0(\eps)$ be given by the large periods property. 

We claim that $f^n(V)\cap U\neq\emptyset$, for every $n\geq N_0$. Indeed, we know that there exists $p\in\fix(f^{\tau})$, with $\tau=n+n_1$, whose orbit under $f$ is $\eps$ dense in $X$. By the choice of $\eps$, there is an iterate of $p$ in $B$. Since $p$ is periodic we shall assume for simplicity that $p$ itself is in $B$. This implies that $f^{n_1}(p)\in V$, and therefore
$$f^n(f^{n_1}(p))=f^{n+n_1}(p)=f^{\tau}(p)=p\in U.$$
This proves our claim, and establishes the lemma.
\end{proof} 

It is a natural question if the converse of this result is true. However, Carvalho and Kwietniak \cite{CK}
gave an example of a homeomorphism of a compact metric space with the two-sided limit shadowing property, but
without  periodic points. Theorem B in \cite{CK} establishes that the
two-sided limit shadowing
property implies topological mixing. Therefore, the converse of Lemma \ref{criteriopramixing}
is not true in general.

We now turn our attention to the
differentiable setting and the semi-local dynamics of homoclinic classes.

\begin{definition}
Let $f:M\to M$ be a diffeomorphism and let $H(p)$ be a homoclinic class of $f$. We say
that an invariant subset $\Lambda\subset H(p)$ has the \emph{homoclinic large periods property} if for any $\eps>0$
there exists $N_0\in\N$ such that for every $n\geq N_0$ it is possible to find a point $p_n\in\fix(f^n)$ in $\Lambda$,
and homoclinically related with $p$, 
whose orbit under $f$ is $\eps$ dense in $\Lambda$.
\end{definition}


In the sequel, we shall establish a
result which produces hyperbolic horseshoes having the homoclinic large periods property when 
there exists a special type of homoclinic intersection.

For its proof we shall need the classical shadowing lemma.

\begin{definition}
Let $f:X\to X$ be a homeomorphism of a metric space $X$. Given $\delta>0$ we say that a sequence $\{x_n\}$ is a $\eps$-pseudo orbit if $d(f(x_n),x_{n+1})<\eps$, for every $n$. We say that the pseudo orbit is $\eps$ shadowed by a point $x\in X$, for $\eps>0$, if $d(f^n(x),x_n)<\eps$, for every $n$. The pseudo orbit is said to be periodic if there exists a minimum number $\tau$ such that $x_{n+\tau}=x_n$, for every $n$. The number $\tau$ is called the period of the pseudo orbit.  
\end{definition}

\begin{theorem}[Shadowing Lemma \cite{Rob}]
\label{t.shadowinglemma}
Let $\Lambda$ be a locally maximal hyperbolic set. For every $\eps>0$ there exists a $\delta>0$ such that every periodic $\delta$-pseudo orbit can be $\eps$-shadowed by a periodic orbit. Moreover, if $\tau$ is the period of the pseudo orbit, then the periodic point is a fixed point for $f^{\tau}$.
\end{theorem}

\begin{lemma}
\label{horseshoemixing}
Let $f$ be a diffeomorphisms with a hyperbolic periodic point $p$ such that there exists a point of transverse intersection $q\in W^s(p)\tcap W^u(f(p))$. Then, for any small enough neighborhood $U$ of $O(p)\cup O(q)$, the restriction of $f$ to the maximal invariant set $\Lambda_U=\cap_{n\in\Z}f^n(U)$ has the homoclinic large periods property.
\end{lemma}
\begin{proof}
For this proof, we denote $\tau:=\tau(p)$ the period of $p$.

It is a well known result (see for instance, Theorem 4.5, pg. 260 in \cite{Rob}) that for any small enough neighborhood $U$ of $O(p)\cup O(q)$ the maximal invariant set $\Lambda_U=\cap_{n\in\Z}f^n(U)$ is a hyperbolic set. 

Take an arbitrary $\eps>0$ and $\delta>0$ given by Theorem \ref{t.shadowinglemma}. 
We claim that there exists a number $N_0$ such that for every $n\geq N_0$ it 
is possible to construct a periodic $\delta$-pseudo orbit inside $U$, with period 
exactly equal to $n$, and whose Hausdorff distance to $O(p)\cup O(q)$ is smaller 
than $\eps$.

Once we have settled this, the shadowing lemma will produce periodic orbits, which
are fixed points for $f^n$, and whose Hausdorff distance to $O(p)\cup O(q)$ is $2\eps$.
In particular, these orbits must be $3\eps$ dense
in $\Lambda_U$, with respect to the Hausdorff distance.
Moreover, if $\eps$ is small enough, all of these periodic orbits will be homoclinically related 
by the hyperbolicity of $\Lambda_U$.

Thus, we are left to show our claim.
With such goal in mind, we take a large iterate $x=f^{N\tau}(q)$ such that $$f^{-r\tau}(x)\in B(p,\delta/2),$$ for every $r=0,...,\tau-1$. Observe that $f^{-1}(x)\in W^u(p)$, since $q\in W^u(f(p))$. This implies that there exists a smallest positive integer $l\in\N$ such that $$f^{-l\tau-1}(x)\in B(p,\delta/2).$$

Now, we can give the number $N_0$. For each $r=1,...,\tau-1$, let $k_r=rl$ and take $L=\prod_{r=1}^{\tau-1}k_r$. We define $N_0:=L\tau$. Observe that if $n\geq N_0$ we can write $$n=(a+L)\tau+r=(a+L-k_r)\tau+k_r\tau+r,$$ for some $r\in\{1,...,\tau-1\}$ and $a\in\N$.

To complete the proof, we shall give the pseudo orbit. It will be given by several strings of orbit, with jumps at specific points. For this reason, and for the sake of clarity, we divide the construction in several steps between each jump. 
\begin{itemize}
\item \emph{The first string:} Define $x_0=f^{-(l+r)\tau-1}(x)$, $x_{j}=f^j(x_0)$, for every $j=1,...,l\tau$.
\item \emph{The second string:} Notice that $f(x_{l\tau})=f^{-r\tau}(x)\in B(p,\delta/2)$. Put $x_{l\tau+1}=f^{-(l+r-1)\tau-1}(x)\in B(p,\delta/2)$, and $x_{l\tau+1+j}=f^j(x_{l\tau+1})$, for every $j=1,...,l\tau$.
\item \emph{The procedure continues inductively:} Notice again that $f(x_{2l\tau+1})=f^{-(r-1)\tau}(x)$ $\in B(p,\delta/2)$, and put $x_{2l\tau+2}=f^{-(l+r-2)\tau-1}(x)$. We proceed with the construction in an analogous way, defining $x_{jl\tau+j}:=f^{-(l+r-j)\tau-1}(x)$ and the next $l\tau$ terms of the sequence as simply the iterates of this point, for every $j<r$. In this manner we construct a sequence with $rl\tau+r-1$ terms.
\item \emph{The last string:} Observe that $f(x_{rl\tau+r-1})=f^{-\tau}(x)\in B(p,\delta/2)$. Hence, we can choose $x_{rl\tau+r}=x$ and the next $(a+L-k_r)\tau-1$ terms of the sequence as simply the iterates of this point, all of which belongs to $B(p,\delta/2)$.
\item \emph{The last jump:} Finally, we close the pseudo orbit by putting $x_{(a+L-k_r)\tau+k_r\tau+r}=x_0$.
\end{itemize}
This gives a periodic $\delta$-pseudo orbit with period $n$, as required.    
\end{proof} 

As an application, from Lemmas \ref{criteriopramixing} and \ref{horseshoemixing} we obtain the following result. 

\begin{proposition}
\label{mixinghorseshoe}
 Let $f$ be a diffeomorphisms with a hyperbolic periodic point $p$ having a non empty transversal intersection between its stable manifold  and the unstable manifold of $f(p)$, i.e. there exists $q\in W^s(p,f)\tcap W^u(f(p),f)$. Then, for any small enough neighborhood $U$ of $O(p)\cup O(q)$, the maximal invariant set $\Lambda_U$ in $U$ is topologically mixing hyperbolic set. 
\end{proposition}

 As a by product of these arguments, we prove that if a homoclinic class 
have the homoclinic large periods property then this holds robustly.  

\begin{proposition}
\label{r.lpp}
Let $f$ be a diffeomorphisms with a hyperbolic periodic point $p$ such that the homoclinic class of $p$, $H(p)$, has the homoclinic large periods property. Then, $H(p(g))$ has the homoclinic large periods property for any diffeomorphism $g$ close enough to $f$. 
\end{proposition}

\begin{proof}[Proof of Proposition \ref{r.lpp}]
Since $H(p)$ has the homoclinic large periods property the period of this homoclinic class has to be one,  $l(O(p))=1$. Indeed, unless the class reduce itself to a fixed point, there will be two periodic points homoclinically related to $p$ such that their periods are two distinct prime numbers. Hence, by Proposition \ref{prop.bla} we have that $W^s(p)\tcap W^u(f(p))\neq\emptyset$. Therefore, since this intersection is robust, we can conclude also by Proposition \ref{prop.bla} and Remark \ref{rmk.int.} that
$W^s(\tilde{p})\tcap W^u(g(\tilde{p}))\neq\emptyset$ for every hyperbolic period point $\tilde{p}$ homoclinically related to $p(g)$, for every diffeomorphism $g$ close enough to $f$. 

So, take an arbitrary $\eps>0$. There exists a periodic point $\tilde{p}\in H(p(g))$, homoclinically related 
with $p(g)$ and whose orbit is $\eps/2$ dense in $H(p(g))$.  Now, Lemma \ref{horseshoemixing} implies that there exists $N_0$ 
such that for every $n\geq N_0$ we can find a periodic orbit $\gamma=O(b)$ homoclinically related to $\tilde{p}$, $b\in\fix(g^n)$, which contains a subset  $\eps/2$ close to $O(\tilde{p})$ in the Hausdorff distance. 
In particular, $\gamma$ is an $\eps$ dense orbit inside $H(p(g))$. This establishes that $H(p(g))$ has the homoclinic large periods property, and completes the proof.
\end{proof}

Observe that the above proof establishes indeed that if a 
homoclinic class $H(p)$ of a diffemorphism $f$ is such that $W^s(p)\tcap W^u(f(p))\neq\emptyset$ 
then $H(p)$ has the homoclinic large periods property. Thus, combining these facts and  Theorem \ref{t.acmixing} we have the following corollary.

\begin{corollary}
Let $f$ be a generic diffeomorphism. An isolated homoclinic class of $f$ is topologically 
mixing if, and only if, it has homoclinic large periods property robustly. 
\label{c.lpp}\end{corollary}

\section{Topologically mixing homoclinic classes}

\subsection{Denseness of Bernoulli measures: Proof of Theorem \ref{teoA}}

We recall the following result of Bowen.

\begin{theorem}[\cite{Bow1}, Theorem 34]
\label{bowen}
Let $\Lambda$ be a topologically mixing isolated hyperbolic set. Then, there exists a Bernoulli measure supported in $\Lambda$. 
\end{theorem}

\begin{remark}
Actually Bowen constructs a measure such that $(f|_{\Lambda},\mu_B)$ is a $K$-automorphism. But, in this case, $(f_{\Lambda},\mu_B)$ is measure theoretically isomorphic to a mixing Markov chain and by \cite{FO} it is isomorphic to a Bernoulli shift. 
\end{remark}

Now, we give the proof of Theorem \ref{teoA}.

\begin{proof}[Proof of Theorem \ref{teoA}]
Let $H(p)$ be an isolated topologically mixing homoclinic class of a $C^1$ generic diffeomorphism 
$f$. Let $\mu$ be an invariant measure supported in $H(p)$ and let $\eps>0$ be arbitrarily 
chosen. By Theorem \ref{ABC} there exists a measure $\mu_{\tilde{p}}$, supported on a hyperbolic
periodic orbit $O(\tilde{p})$, with $\tilde{p}\in H(p)$, which is $\eps/2$ close to $\mu$. 

 Since $f$ is $C^1$ generic, Theorem \ref{t.cmp} implies
that $H(\tilde{p})=H(p)$. In particular, we have that $H(\tilde{p})$ is topologically mixing.

From Remark \ref{mixing1} we know that there exists a point $q\in W^s(\tilde{p})\tcap W^u(f(\tilde{p}))$. 
For every small neighborhood $U$ of $O(\tilde{p})\cup O(q)$, Proposition \ref{mixinghorseshoe} tells us 
that the maximal invariant set $\Lambda_U=\cap_{n\in\Z}f^n(U)$ is a topologically mixing hyperbolic set. 
Moreover, since $q$ is a homoclinic point of $\tilde{p}$, by choosing $U$ sufficiently
small we have that
the points in $\Lambda_U$ spent portions of their orbit as large as we please shadowing the orbit of $\tilde{p}$.

Now, take $\nu$ the Bernoulli measure supported in $\Lambda_U$ which is given by Theorem \ref{bowen}. Since a typical point in the support of $\nu$ spent large portions of its orbit shadowing the orbit of $\tilde{p}$, we can choose $U$ such that $\nu$ is $\eps/2$ close to $\mu_{\tilde{p}}$. 

Thus, $\nu$ is $\eps$ close to $\mu$ and we are done.
\end{proof} 

\begin{remark}
The techniques employed above can be used to give a new proof of Sigmund's result on 
the denseness of Bernoulli measures for hyperbolic topologically mixing basic sets \cite{S1}. 
Indeed, our use of the large periods property gives a geometric alternative to the symbolic approach of Sigmund and a proof of his result using our techniques would proceed by the same argument as above, in the proof of Theorem \ref{teoA},  after we modified the following key results: first, Sigmund's result on denseness of periodic measures in a hyperbolic basic set, \cite{S2}, can be used instead of Theorem \ref{ABC}, and second Bowen's proof  of Smale's Spectral Decomposition Theorem (see pag. 47 of \cite{Bow}) can be used instead of Theorem \ref{t.acmixing} and Remark \ref{mixing1} to show the existence of nice intersections between the stable and unstable manifolds of hyperbolic periodic points. Therefore, with these modifications, the same proof as above can be applied.  
\end{remark}

\section{Robustly large Homoclinic class}

In this section we shall prove Theorem \ref{main.homoclinic} as a consequence of the following result:

\begin{theorem}
Let $f\in \diff^1(M)$ be a robustly transitive strong partially hyperbolic diffeomorphism, with 
$TM=E^s\oplus E^c_1\oplus \ldots E^c_k\oplus E^u$, having hyperbolic periodic points $p_s$ and 
$p_u$ with index $s$ and $d-u$, respectively,  where $s=dim\ E^s$ and $u=dim \ E^u$. 
Then, there exists an open subset $\mathcal{V}_f$ whose closure contains $f$, such that 
$M=H(p_s(g))=H(p_u(g))$ for every $g\in \mathcal{V}_f$. 
\label{r.homoclinic}\end{theorem} 

Before we prove Theorem \ref{r.homoclinic}, let us see how it implies Theorem \ref{main.homoclinic}.

\begin{proof}[Proof of Theorem \ref{main.homoclinic}]
First we observe that it suffices to deal with the interior of non-hyperbolic robustly transitive 
diffeomorphisms, since in the Anosov case the whole manifold is robustly a homoclinic class, which is a consequence of the shadowing lemma. 

Recall that $\mathcal{T}_{NH}(M)\subset\ST(M)$ denotes the interior of non-hyperbolic robustly transitive diffeomorphisms far from homoclinic tangencies. Hence, by Theorem \ref{BC} and Theorem \ref{p.CSY} there exists a residual subset $\mathcal{R}$ in $\mathcal{T}_{NH}(M)$ such that if $f\in \mathcal{R}$ then:
\begin{itemize}
\item[a)] $M$ coincides with a homoclinic class;

\item[b)]  $f$ is partially hyperbolic, with the central bundle admitting a splitting in one dimension sub bundles. I.e., $TM=E^s\oplus E_1^c\oplus \ldots \oplus E^c_k\oplus E^u$;

\item[c)] either there exist a hyperbolic periodic point with index $s$, or there exists hyperbolic periodic points  
with index $s+1$ whose  the $(s+1)-$Lyapunov exponent is arbitrarilly close to zero. 
Where $s=dim\ E^s$. 

\item[d)] either there exist a hyperbolic periodic point with index $d-u$, or there exists hyperbolic periodic points  with index $d-u-1$ whose  the $(d-u-1)-$Lyapunov exponent is arbitrary close to zero. Where $u=dim \ E^u$. 
\end{itemize}

According to Theorem \ref{p.CSY}, $E^s$ and/or $E^u$ could be trivial. However, this 
cannot happen in our situation. Indeed, we claim that both $E^s$ and $E^u$ are non-trivial. In particular, $f$ is strongly partially hyperbolic. To see this,
suppose by contradiction the existence of $f\in \mathcal{R}$ with $E^s$ trivial. 
Hence, by item $c)$ above, $f$ should have either a source or hyperbolic periodic points with index one, 
with the only one Lyapunov negative exponent being arbitrary close to zero. In the last case, we can use 
Lemma \ref{change.index} to perturb  $f$ in order to find also a source. Therefore, if $E^s$ is 
trivial, then we can find a diffeomorphism $g$ close to $f$, having a source, which is a contradiction 
with the transitivity of $g$. Similarly we conclude that $E^u$ is also non-trivial. 
Henceforth, item b) above can be replaced by:

\begin{itemize}
\item[b')] every $f\in \mathcal{R}$ is strongly partially hyperbolic. 
\end{itemize}

Moreover, by the same argument above using Lemma \ref{change.index}, after a perturbation we can
assume that $f$ has hyperbolic periodic points of indices $s$ and $d-u$. Thus, we can find a dense 
subset $\mathcal{R}_1$ inside $\mathcal{T}_{NH}(M)$ formed by robustly transitive strong partially hyperbolic diffeomorphisms $f$ satisfying the hypothesis of Theorem \ref{r.homoclinic}. Then, considering $\mathcal{V}_f$ given by Theorem \ref{r.homoclinic} for every $f\in \mathcal{R}_1$ we have that
$$
\mathcal{A}=\bigcup_{f\in \mathcal{R}_1} \mathcal{V}_f,
$$
is an open and dense subset of $\mathcal{T}_{NH}(M)\subset\ST(M)$.  
By Theorem \ref{r.homoclinic}, for every 
diffeomorphism in $\mathcal{A}$ the whole manifold $M$ coincides with a homoclinic 
class. This ends the proof 
\end{proof}

In the sequence we prove some technical results which are key steps in the proof of Theorem \ref{r.homoclinic}.

The following result allows to find open sets of diffeomorphisms for which the topological 
dimension of stable (and unstable manifold) of hyperbolic periodic points is larger than the 
differentiable dimension.  

\begin{lemma} Let $f\in \diff^1(M)$ be a robustly transitive strong partially hyperbolic  diffeomorphism.  Suppose there are hyperbolic periodic points $p_j$, $j=i,\ i+1, \ldots, k$, with indices $I(p_j)=j$ for $f$. Hence, given any small enough neighborhood  $\mathcal{U}$ of $f$, where is defined the continuation of the hyperbolic periodic points $p_j$, there exists an open set $\mathcal{V}\subset \mathcal{U}$ such that for every $g\in \mathcal{V}$: 
$$
W^s(p_k(g))\subset cl(W^s(p_{k-1}(g)))\subset \ldots \subset cl(W^s(p_{i+1}(g)))\subset cl(W^s(p_{i}(g))), and
$$   
$$
W^u(p_i(g))\subset cl(W^u(p_{i+1}(g)))\subset \ldots \subset cl(W^u(p_{k-1}(g)))\subset cl(W^u(p_{k}(g))).
$$   
\label{l.blender}\end{lemma}

To prove the above lemma  we will use the following result which is a consequence of Proposition 6.14 and 
Lemma 6.12 in \cite{BDV}, which are results of Diaz and Rocha \cite{DR}. It is worth to point out that this result is a consequence of the well known blender technique, which appears by means of unfolding a heterodimensional co-dimensional one cycle far from homoclinic tangencies. 

\begin{proposition}
Let $f$ be a $C^1$ diffeomorphism with a heterodimensional cycle associated to saddles $p$ and $\tilde{p}$ with indices $i$ and $i+1$, respectively. Suppose that the cycle is $C^1-$far from homoclinic tangencies. Then there exists an open set $\mathcal{V}\subset \diff^1(M)$ whose closure contains $f$ such that for every $g\in \mathcal{V}$ 
$$
W^s(\tilde{p}(g))\subset cl(W^s(p(g))) \text{ and } W^u(p(g))\subset cl(W^u(\tilde{p}(g))). 
$$
\label{p.blender}\end{proposition}

\begin{proof}[Proof of Lemma \ref{l.blender}]
Since $f$ is a robustly transitive strong partially hyperbolic diffeomorphism, we can assume that every diffeomorphism $g\in \mathcal{U}$ is transitive and is strong partially hyperbolic, reducing  $\mathcal{U}$ if necessary. In particular,  $\mathcal{U}$ is far from homoclinic tangencies, $\mathcal{U}\subset (cl(\mathcal{HT}(M)))^c$. Now, using the transitivity of $f$, there are points $x_n$ converging to the stable manifold of $p_{i+1}$ whose a sequence of iterates $f^{m_n}(x_n)$ is converging  to the unstable manifold of $p_{i}$. Hence, we can use Hayashi's connecting lemma, to perturb the diffeomorphism $f$ to $\tilde{f}$ such that $W^u(p_i(\tilde{f}))$ intersects $W^s(p_{i+1}(\tilde{f}))$, which one we could assume be transversal after a  perturbation, if necessary, since $\dim\ W^u(p_i(\tilde{f}))+\dim\ W^s(p_{i+1}(\tilde{f}))>d$. Hence, we can use once more the connecting lemma to find $f_1\in {\mathcal{U}}$ close to $\tilde{f}$ exhibiting a heterodimensional cycle between $p_i(f_1)$
and $p_{i+1}(f_1)$, since $\tilde{f}$ is also transitive.  Moreover, and in fact this is needed to apply Proposition \ref{p.blender}, the intersection between $W^s(p_i(f_1))$ and $W^u(p_{i+1}(f_1))$ could be assumed quasi-transversal in the sense that $ T_qW^s(p_i(f_1))\cap T_qW^u(p_{i+1}(f_1))=\{0\}$. If this is not true, we can do a perturbation of the diffeomorphism using Franks lemma, to get such property.   

Thus, since $f_1$ is far from homoclinic tangencies, we can use Proposition \ref{p.blender}  to find  an open set $\mathcal{V}_1\subset \mathcal{U}$ such that 
$$
W^s(p_{i+1}(g))\subset cl(W^s(p_i(g))) \text{ and } W^u(p_{i}(g))\subset cl(W^u(p_{i+1}(g))), 
$$
for every $g\in \mathcal{V}_1$. 

Now, since $f_1$ is also  robustly transitive we can repeat the above argument to find $f_2\in \mathcal{V}_1$ exhibiting a heterodimensional cycle between $p_{i+1}$ and $p_{i+2}$. Thus, by Proposition \ref{p.blender} there exists an open set $\mathcal{V}_2\subset \mathcal{V}_1$, such that 
$$
W^s(p_{i+2}(g))\subset cl(W^s(p_{i+1}(g))) \text{ and } W^u(p_{i+1}(g))\subset cl(W^u(p_{i+2}(g))), 
$$
for every $g\in \mathcal{V}_2$. 

Repeating this argument finitely many times we will find open sets $\mathcal{V}_{k-i}\subset \mathcal{V}_{k-i-1}\subset \ldots \subset \mathcal{V}_1$ such that 
$$
W^s(p_{i+j}(g))\subset cl(W^s(p_{i+j-1}(g))) \text{ and } W^u(p_{i+j-1}(g))\subset cl(W^u(p_{i+j}(g))), 
$$
for every $g\in \mathcal{V}_{j}$, and $j=1,\dots k-i$. 

Taking $\mathcal{V}=\mathcal{V}_{k-i}$ the result follows. 
\end{proof}

The next result use properties of a partially hyperbolic splitting to guarantee that some special kind of dense sub-manifolds in $M$ should intersect each other transversally and densely in the whole manifold. 

\begin{lemma}
Let $f$ be a partially hyperbolic diffeomorphism on $M$ with non trivial stable bundle $E^s$, and having a hyperbolic periodic point $p$ with index $s=dim \ E^s$. If  $W^s(O(p))$ and $W^u(O(p))$ are dense in $M$, then $M=H(p)$.  
\label{transversallity}\end{lemma}

\begin{proof}
Let $E^s\oplus E^c\oplus E^u$ be the partially hyperbolic splitting. Using Remark \ref{r.hps} we know that the local strong stable manifolds have uniform size. 

For any $x\in M$, since $W^u(O(p))$ is dense, there exists $q\in W^u(O(p))$ arbitrarily close to $x$. Also, by hypothesis of the index of $p$, and the partially hyperbolic structure, it should be true that $T_qW^u(O(p))=E^c\oplus E^u$. Hence, by the continuity of the local strong stable manifold,  $W^{ss}_{loc}(y)$  should intersect transversally $W^u(O(p))$ in a point close to $q$, for any point $y$ close enough to $q$. In particular, since $W^s(O(p))$ is also dense, there exists $\tilde{q}\in W^s(O(p))$ such that $W_{loc}^{ss}(\tilde{q})$ intersects transversally $W^u(O(p))$. However, $W_{loc}^{ss}(\tilde{q})$ is contained in $W^s(O(p))$, which implies there is a transversal intersection between $W^s(O(p))$ and $W^u(O(p))$ close to $q$, in particular, close to $x$.    
\end{proof}

Finally, using the above lemmas we give a proof of Theorem \ref{r.homoclinic}.

\begin{proof}[Proof Theorem \ref{r.homoclinic}]
Since $p_s$ and $p_u$ are hyperbolic periodic points, we take $\mathcal{U}$ small enough such that every diffeomorphism  $g\in \mathcal{U}$ has defined the continuations $p_s(g)$ and $p_u(g)$. Reducing $\mathcal{U}$ if necessary, we could also assume that every $g\in \mathcal{U}$ is a strong partially hyperbolic diffeomorphism with same extremal bundles dimension as in the partially hyperbolic decomposition of $TM$ as $f$, which follows by the continuity of the partially hyperbolicity and the existence of $p_s$ and $p_u$ robustly. 

Now, using Theorem \ref{BC} together with Theorem \ref{t.ABCDW} we can find a residual subset $\mathcal{R}$ in $\mathcal{U}$ such that $M$ coincides with a homoclinic class for every $g\in \mathcal{R}$, and moreover $g$ has hyperbolic periodic points of any index in $[s,d-u]\cap\N$.  

We fix $g\in \mathcal{R}$, and let $p_s=p_s(g)$, $p_{s+1}$, $\ldots$, $p_{d-u}=p_u(g)$ be hyperbolic periodic points of $g$ with indices $s$, $s+1$, $\ldots$, $d-u$, respectively. Also, for all $n\in \N$, let $\mathcal{V}_n\subset \mathcal{U}$  small neighborhoods of $g$, such that if $g_n\in \mathcal{V}_n$, then $g_n$ converges to $g$ in the $C^1-$topology, when $n$ goes to infinity. 

Now, since $g$ is still a robustly transitive strong partially hyperbolic diffeomorphism having hyperbolic periodic points of all possible indices, we denote by $\tilde{\mathcal{V}}_n\subset \mathcal{V}_n$ the open sets given for $g$ and $\mathcal{V}_n$ by Lemma \ref{l.blender}. Hence, using the invariance of the stable  manifold of hyperbolic periodic points, by Lemma \ref{l.blender} we have the following: 
\begin{equation}
cl(W^s(O(p_{d-u}(r))))\subset cl(W^s(O(p_{d-u-1}(r))))\subset \ldots \subset cl(W^s(O(p_{s}(r)))),  
\label{eq10}\end{equation}
for every $r\in \tilde{\mathcal{V}}_n$.

\vspace{0,2cm}
{\it Claim: $W^u(O(p_{s}(r)))$ and $W^s(O(p_{d-u}(r)))$ are dense in $M$, for every $r\in \tilde{\mathcal{V}}_n$}. 

\vspace{0,1cm}
Since $r$ is transitive, there exist $x\in M$ such that the forward orbit of $x$ is dense in $M$. Now, since $r$ is partially hyperbolic, for Remark \ref{r.hps} there exists the strong stable foliation that integrates the direction $E^{s}$. Moreover, these leafs have local uniform length. Hence, as done in the proof of Lemma \ref{transversallity}, we can take $r^j(x)$ close enough to $p_s(r)$ such that $W^{ss}(x)$, the strong stable leaf containing $x$, intersects the local unstable manifold of $p_s(r)$, $W^u_{loc}(p_s(r))$.  Therefore, since points in the same strong stable leaf have the same omega limit set, we have that $W^u(O(p_s(r)))$ is dense in the whole manifold $M$. We can repeat this argument using also the existence of a point $y$ having a dense backward orbit, and the existence of the strong unstable foliation to conclude that $W^s(O(p_{d-u}(r)))$ is also dense in $M$. 

\vspace{0,2cm}

Thus, by equation (\ref{eq10}) and the Claim, we have that $W^s(O(p_{s}(r))))$ is dense in $M$. 
Similarly, we can show that  $W^u(O(p_{d-u}(r))))$ is also dense in $M$. 

Provided that $r$ is strong partially hyperbolic, and that $W^s(O(p_{i}(r))))$ and \\ $W^u(O(p_{i}(r))))$ are dense in $M$, for $i=s$ and $d-u$, we can apply Lemma \ref{transversallity} for $f$ and $f^{-1}$ to conclude that 
$$
M=H(p_{s}(r))=H(p_{d-u}(r)),
$$
for every $r\in \tilde{\mathcal{V}}_n$. 

Hence, the proof is finished defining $\tilde{\mathcal{V}}_g=\cup \tilde{\mathcal{V}}_n$, and 
$$
\mathcal{V}_f=\bigcup_{g\in \mathcal{R}} \tilde{\mathcal{V}}_g,
$$
which is  an open and dense subset of $\mathcal{U}$, and hence contains $f$ in its closure. 
\end{proof}

\end{document}